\documentclass[reqno,centertags, 12pt]{amsart}
\usepackage{amsmath,amsthm,amscd,amssymb}
\usepackage{latexsym}
\newcommand{\bbC}{{\mathbb{C}}}
\newcommand{\bbD}{{\mathbb{D}}}
\newcommand{\bbE}{{\mathbb{E}}}

\newcommand{\bbP}{{\mathbb{P}}}

\newcommand{\bbR}{{\mathbb{R}}}

\newcommand{\bbV}{{\mathbb{V}}}
\newcommand{\bbZ}{{\mathbb{Z}}}

\newcommand{\calR}{{\mathcal R}}

\newcommand{\lb}{\label}
\newcommand{\f}{\frac}

\newcommand{\ol}{\overline}
\newcommand{\ti}{\tilde  }
\newcommand{\wti}{\widetilde  }

\newcommand{\var}{\text{\rm{Var}}}

\newcommand{\ac}{\text{\rm{ac}}}

\newcommand{\supp}{\text{\rm{supp}}}

\newcommand{\bi}{\bibitem}

\newcommand{\beq}{\begin{equation}}
\newcommand{\eeq}{\end{equation}}
\newcommand{\ba}{\begin{align}}
\newcommand{\ea}{\end{align}}
\newcommand{\veps}{\varepsilon}

\newcommand{\fre}{\frak{e}}

\newcommand{\abs}[1]{\lvert#1\rvert}

\newcommand{\norm}[1]{\lVert#1\rVert}

\newcounter{smalllist}
\newenvironment{SL}{\begin{list}{{\rm\roman{smalllist})}}{%
\setlength{\topsep}{0mm}\setlength{\parsep}{0mm}\setlength{\itemsep}{0mm}%
\setlength{\labelwidth}{2em}\setlength{\leftmargin}{2em}\usecounter{smalllist}%
}}{\end{list}}

\allowdisplaybreaks
\numberwithin{equation}{section}

\newtheorem{theorem}{Theorem}[section]

\newtheorem{lemma}[theorem]{Lemma}
\newtheorem{corollary}[theorem]{Corollary}
\theoremstyle{definition}
\newtheorem{example}[theorem]{Example}
\newtheorem*{ex}{Example}

\theoremstyle{remark}
\newtheorem*{remark}{Remark}
\newtheorem*{remarks}{Remarks}

\begin{document}

\title{Natural Boundaries and Spectral Theory}
\author[J.~Breuer and B.~Simon]{Jonathan Breuer$^1$ and Barry Simon$^2$}

\thanks{$^1$ Einstein Institute of Mathematics, The Hebrew University of Jerusalem,
Jerusalem 91904, Israel. E-mail: jbreuer@math.huji.ac.il}
\thanks{$^2$ Mathematics 253-37, California Institute of Technology, Pasadena, CA 91125, USA.
E-mail: bsimon@caltech.edu. Supported in part by NSF grant DMS-0652919}

\date{August 8, 2010}

\begin{abstract} We present and exploit an analogy between lack of absolutely continuous spectrum for Schr\"odinger operators
and natural boundaries for power series. Among our new results are generalizations of Hecke's example and natural
boundary examples for random power series where independence is not assumed.
\end{abstract}

\maketitle

\section{Introduction} \lb{s1}

In this paper, we'll present and exploit a powerful analogy between spectral theory and the question of
when a power series $f(z)=\sum_{n=0}^\infty a_n z^n$ defining an analytic function on $\bbD=\{z\mid\abs{z}<1\}$ has
a natural boundary on $\partial\bbD$, in that for no $z_0=e^{i\theta}$ does $f$ have an analytic continuation to $\{z\mid
\abs{z-z_0}<\delta\}$ for some $\delta >0$. In particular, we shall import two notions (``reflectionless'' and ``right
limit'') from the spectral theory of Jacobi matrices and obtain a general theorem regarding the consequence of the
possibility to analytically continue $f$ across an arc. While spectral theory ideas motivated our approach to natural
boundaries, what we develop doesn't require any spectral theory. The reader not knowledgeable in spectral theory and not
interested in the background should skip to the paragraph containing \eqref{1.4}.

As it turns out, some of the results presented in this paper are not new. 
In particular, in \cite{Agmon1} Agmon treated natural boundaries for general series with  radius of convergence $1$ using a related approach. Our Theorem 
\ref{T1.3} is a special case of his results. To the best of our knowledge, however, our treatment of \emph{strong} natural boundaries using this approach 
(see Theorem \ref{T1.4} below) is completely new.

Of course, since the relevant spectral theoretic notions were not yet defined at the time of publication of \cite{Agmon1}, Agmon was unaware of 
the analogy. Surprisingly, \cite{Agmon1} is rarely quoted in the relevant literature and seems to be little known. 
Thus, a secondary aim of our paper is to draw attention to Agmon's work. 

Moreover, we believe the significance of the analogy presented in this paper is 
more than merely anecdotal. As an example, our knowledge of the applications of these ideas to random potentials in the spectral theoretic setting, led us 
to apply these methods to random series. In this context we have obtained results for series where only a subsequence is random and for general bounded ergodic nondeterministic series. Another example is Theorem \ref{T1.8} that has Hecke's famous example \eqref{1.13} as a special case.

Thus, Theorems \ref{T1.3} and \ref{T1.4} provide a general framework in which many existing and new results concerning natural boundaries can be derived. 

To set the stage, consider a Jacobi matrix,
\begin{equation} \lb{1.1}
J=
\begin{pmatrix}
b_1 & a_1 & 0  & \cdots \\
a_1 & b_2 & a_2  & \cdots \\
0 & a_2 & b_3  & \cdots \\
\vdots & \vdots & \vdots  & \ddots
\end{pmatrix}
\end{equation}
where $\{a_n,b_n\}_{n=1}^\infty$ are bounded. By a right limit of $J$, we mean a two-sided Jacobi matrix, $J^{(r)}$, with
parameters $\{a_n^{(r)}, b_n^{(r)}\}_{n=-\infty}^\infty$ given by
\begin{equation} \lb{1.2}
a_n^{(r)} = \lim_{k\to\infty}\, a_{n+n_k} \qquad b_n^{(r)} = \lim_{k\to\infty}\, b_{n+n_k}
\end{equation}
where $n_k\to\infty$ is a subsequence. By compactness, right limits exist. This definition is from Last--Simon \cite{LS},
who used it in

\begin{theorem}[\cite{LS}] \lb{T1.1} If $\Sigma_\ac(A)$ is the essential support of the a.c.\ spectrum of an operator, $A$, then
\begin{equation} \lb{1.3}
\Sigma_\ac(J) \subset \Sigma_\ac(J^{(r)})
\end{equation}
for every right limit  $J^{(r)}$ of $J$.
\end{theorem}

Right limits are also relevant to essential spectrum where they have been exploited by several authors (e.g.,
\cite{CWL08,CWLppt,GI1,LS,S304,M1,Rab05}), but their relevance for a.c.\ spectrum goes back to Last--Simon \cite{LS}.
Recently, Remling \cite{Rem} found a much stronger property of right limits when there is a.c.\ spectrum. It depends on the
notion of reflectionless two-sided Jacobi matrix. The precise definition is irrelevant for our discussion here (see, e.g.,
\cite{BRS,Rem,Rice,SY}), but we note that if $\ti J$ is a two-sided Jacobi matrix, reflectionless on some $\fre\subset\bbR$
with the Lebesgue measure, $\abs{\fre}$,  of $\fre$ nonzero, then $\fre\subset\Sigma_\ac (\ti J)$ and $\{\ti a_n,
\ti b_n\}_{n=-\infty}^{-1}$ determine $\{\ti a_n, \ti b_n\}_{n=0}^\infty$. Remling \cite{Rem} proved

\begin{theorem}[\cite{Rem}] \lb{T1.2} If $\fre$ is $\Sigma_\ac(J)$ for a Jacobi matrix, then any right limit, $J^{(r)}$, is
reflectionless on $\fre$.
\end{theorem}

With this result, Remling was able to recover and extend virtually every result on the absence of a.c.\ spectrum for situations
where $\{a_n,b_n\}_{n=1}^\infty$ are bounded, and the lack of a.c.\ spectrum is due to a part of the $a$'s and $b$'s that is
dominant at infinity (rather than a perturbation that goes to zero at infinity).

Our work here began by our noticing that the major classes of Jacobi matrices with no a.c.\ spectrum have analogs in
the major classes of results on the occurrence of natural boundaries (see Remmert \cite[Ch.~11]{Remm} for a summary of
classical results on natural boundaries), as seen in
\begin{itemize}
\item Gap theorems \cite{Had,Fabry,Faber,Remm} $\sim$ sparse potentials \cite{Pear78,Rem}
\item Finite-valued power series \cite{Sz1922,Remm} $\sim$ finite-valued Jacobi matrices \cite{Kot89,Rem}
\item Random power series \cite{Stei,PZyg,Kah} $\sim$ Anderson localization \cite{CarLac,PasFig}
\end{itemize}

Let us describe our major abstract results on natural boundaries motivated by Last--Simon \cite{LS} and Remling \cite{Rem}.
Given a power series $f(z)=\sum_{n=0}^\infty a_n z^n$ with
\begin{equation} \lb{1.4}
\sup_n\, \abs{a_n} <\infty
\end{equation}
we define a right limit of $\{a_n\}_{n=0}^\infty$ to be a two-sided sequence $\{b_n\}_{n=-\infty}^\infty$ with
\begin{equation} \lb{1.5}
b_n =\lim_{j\to\infty}\, a_{n+n_j}
\end{equation}
for some $n_j\to\infty$. By compactness and \eqref{1.4}, right limits exist.

Given a two-sided bounded sequence, $\{b_n\}_{n=-\infty}^\infty$, we consider two functions, $f_+(z)$ on $\bbD$ and
$f_-(z)$ on $\bbC\cup\{\infty\}\setminus\ol{\bbD}$, defined by
\begin{equation} \lb{1.6}
f_+(z) =\sum_{n=0}^\infty b_n z^n \qquad f_-(z) =\sum_{n=-\infty}^{-1} b_n z^n
\end{equation}
where the series are guaranteed to converge on the indicated sets. Let $I$ be an open interval in $\partial\bbD$. We say
$\{b_n\}_{n=-\infty}^\infty$ is {\it reflectionless\/} on $I$ if and only if $f_+$ has an analytic continuation from
$\bbD$ to $\bbC\cup\{\infty\}\setminus(\partial\bbD\setminus I)$, so that on $\bbC\cup\{\infty\}\setminus\ol{\bbD}$, we
have that
\begin{equation} \lb{1.6x}
f_+(z)+f_-(z)=0
\end{equation}
Obviously, it suffices that $f_+$ have a continuation to a neighborhood of $I$ so that \eqref{1.6x} holds in the
intersection of that neighborhood and $\bbC\setminus\ol{\bbD}$.

\begin{ex} Let $b_n\equiv 1$. Then $f_+(z)=(1-z)^{-1}$ and $f_-(z)=-(1-z)^{-1}$. This series is reflectionless on
$I=\{e^{i\theta}\mid 0<\theta<2\pi\}$. Similarly, it is easy to see that a periodic $b_n$ of period $p$ is
reflectionless on any $I$ in $\partial\bbD$ with all the $p$-th roots of unity removed.
\end{ex}

Our first main theorem is:

\begin{theorem} \lb{T1.3} Let $f(z)=\sum_{n=0}^\infty a_n z^n$ be a power series with \eqref{1.4}. Suppose $I\subset
\partial\bbD$ is an open interval so that $f(z)$ has an analytic continuation to a neighborhood of $I$. Then every right
limit of $\{a_n\}_{n=0}^\infty $ is reflectionless on $I$.
\end{theorem}

As remarked above, this theorem is a special case of the results in \cite{Agmon1}. While our treatment is limited to bounded sequences and produces interesting results only when $a_n \nrightarrow 0$, Agmon treats general (not necessarily bounded) series. He introduces a way to renormalize the coefficients so that any series with radius of convergence 1 has, after renormalization, nontrivial right limits. He then shows that if the corresponding function has an analytic continuation to an arc on $\partial \bbD$, then the renormalized right limits are reflectionless across that arc. 
For  bounded series, we will actually prove a stronger result (we state Theorem~\ref{T1.3} both for conceptual reasons and because we need it in the proof of the stronger result):

\begin{theorem} \lb{T1.4} Let $f(z)=\sum_{n=0}^\infty a_n z^n$ be a power series with \eqref{1.4}. Suppose $I\subset
\partial\bbD$ is an open interval so that
\begin{equation} \lb{1.7}
\sup_{0<r<1} \int_{e^{i\theta}\in I} \abs{f(re^{i\theta})}\, \f{d\theta}{2\pi} <\infty
\end{equation}
Then every right limit of $\{a_n\}$ is reflectionless on $I$.
\end{theorem}

These theorems imply natural boundaries. We say $f$ has a {\it strong natural boundary\/} on $\partial\bbD$ if
\eqref{1.7} fails for every $I\subset\partial\bbD$. In particular, in that case, $f$ is unbounded in every sector
$\{re^{i\theta}\mid 0<r<1,\, e^{i\theta}\in I\}$.

\begin{corollary}\lb{C1.5} Let $f(z)=\sum_{n=0}^\infty a_n z^n$ be such that for any open interval $I\subset\partial\bbD$,
there is a right limit of $\{a_n\}_{n=0}^\infty$ which is not reflectionless on $I$. Then $f$ has a strong natural boundary
on $\partial\bbD$.
\end{corollary}

Of course, if $a_n\to 0$, it can happen that there are
natural boundaries which are not strong natural boundaries. For example, by the Hadamard gap theorem \cite{Remm},
\begin{equation} \lb{1.8}
f(z)=\sum_{n=1}^\infty \, \f{z^{n!}}{(n!)^n}
\end{equation}
has a natural boundary on $\partial\bbD$ but, for all $k$, $f^{(k)}(z)$ is bounded on $\bbD$.

We note that Duffin--Schaeffer \cite{DuSc} and Boas \cite{Boas} long ago had results on what we call strong natural
boundaries. These results are all for series with gaps or finite-valued series. These authors and also Agmon \cite{Agmon1} have results that prove that certain functions rather than merely having classical natural boundaries are unbounded in every sector (we call this $L^\infty$ natural boundaries).  We note that our approach for going from classical to strong natural boundaries is very close to the method that Agmon \cite{Agmon1} uses to go from classical natural boundaries to unboundedness in every sector.
Theorem \ref{T1.4} is, to the best of our knowledge, the first general theorem concerning strong natural boundaries.  

For us, the point of these theorems is conceptual: they provide a unified framework that makes many theorems transparent.
That said, the following results seem to be new:

\begin{theorem}\lb{T1.6} Suppose $\{a_n\}_{n=0}^\infty$ obeys \eqref{1.4} and there exists $n_j\to\infty$ so that for all
$k<0$,
\begin{align}
\lim_{j\to\infty}\, a_{n_j+k} & =0 \lb{1.9} \\
\liminf\abs{a_{n_j}} &> 0 \lb{1.10}
\end{align}
Then $f(z)=\sum_{n=0}^\infty a_n z^n$ has a strong natural boundary on $\partial\bbD$.
\end{theorem}

\begin{remarks} 1. This result (as well as its extension---Theorem \ref{T4.3}), for $L^\infty$ natural boundaries, appears in Agmon \cite{Agmon1}.

\smallskip
2. The result is true if \eqref{1.9} is replaced by the assumption for all $k>0$ or for all $k<K$\!, some $K<0$ or
for all $k>K$\!, some $K>0$. The proofs are essentially identical. In addition, \eqref{1.9} can be replaced by an
exponential decay condition on the limit; see Theorem~\ref{T4.3}.

\smallskip
3. This includes the famous examples $\sum_{n=1}^\infty z^{n!}$ of Weierstrass and $\sum_{n=0}^\infty z^{n^2}$ of Kronecker.

\smallskip
4. This allows gaps where the set of zeros has zero density. At first sight, this seems a violation of the result of P\'olya
\cite{Polya} and Erd\"os \cite{Erdos} that the Fabry gap theorem is optimal. We'll explain this apparent discrepancy in
Section~\ref{s4}.
\end{remarks}

\begin{theorem} \lb{T1.7} Let $\{a_n(\omega)\}_{\omega\in\Omega}$ be a translation invariant, ergodic, stochastic process,
which is nondeterministic, so that
\begin{equation} \lb{1.11}
\sup_{n,\omega}\, \abs{a_n(\omega)} <\infty
\end{equation}
Then for a.e.\ $\omega$, $\sum_{n=0}^\infty a_n(\omega) z^n$ has a strong natural boundary.
\end{theorem}

\begin{remarks} 1. So far as we know, all previous results on random power series rely on independence and only obtain natural
boundaries, not strong natural boundaries.

\smallskip
2. Recall that a stochastic process $\{a_n(\omega)\}$ is deterministic if $a_0(\omega)$ is a measurable function of
$\{a_n(\omega)\}_{n=-\infty}^{-1}$ for a.e.\ $\omega$, and nondeterministic if it is not deterministic.

\smallskip
3. By taking an average of a deterministic and a nondeterministic process, it is easy to see that ergodicity is essential
for this theorem to be true.
\end{remarks}

\begin{theorem}\lb{T1.8} Let $f\colon\partial\bbD\to\bbC$ be a bounded and piecewise continuous function with only a
finite number of discontinuities, at one of which the one-sided limits exist and are unequal. Then for any irrational
number $q$ and every
$\theta\in\bbR$, we have that
\begin{equation} \lb{1.12}
\sum_{n=0}^\infty f(e^{2\pi i (qn+\theta)}) z^n
\end{equation}
has a strong natural boundary.
\end{theorem}

\begin{remark} This includes Hecke's famous example \cite{Hecke},
\begin{equation} \lb{1.13}
f(z)=\sum_{n=0}^\infty \{nq\} z^n
\end{equation}
where $\{x\}$ is the fractional part of $x\in\bbR$.
\end{remark}

The following is elementary and does not require our full machinery. It was suggested to us by the Wonderland theorem of
Simon \cite{Wonder} and seems to be new.

\begin{theorem} \lb{T1.9} Let $\Omega\subset\bbC$ be a compact set with more than one point. Let $\Omega^\infty$ be a countable
product of copies of $\Omega$ in the weak topology. Then $\{\{a_n\}\}\in\Omega^\infty \mid\sum_{n=0}^\infty a_n z^n$ has a
natural boundary on $\partial\bbD\}$ is a dense $G_\delta$ in $\Omega^\infty$.
\end{theorem}

In Section~\ref{s2}, we prove Theorem~\ref{T1.3}. The key is a lemma of M.~Riesz whose proof we include for completeness. In Section~\ref{s3}, we prove Theorem~\ref{T1.4} using the theory of $H^p$ spaces on
a sector (Duren \cite{Duren}). In Section~\ref{s4}, we discuss gap theorems, including Theorem~\ref{T1.6}. In Section~\ref{s5}, we
discuss Szeg\H{o}'s theorem on finite-valued power series using Theorem~\ref{T1.4}. In Section~\ref{s6}, we discuss random
power series, including Theorem~\ref{T1.7}. In Section~\ref{s7}, we prove Theorem~\ref{T1.8}, following a spectral theory analysis
of Damanik--Killip \cite{DK}. In Section~\ref{s8}, we prove Theorem~\ref{T1.9}. 

We believe our work here opens up numerous new directions in the study of power series and of spectral theory. In particular,
there is a dynamical view of reflectionless in spectral theory (see \cite{BRS}) and there is the distinction in spectral theory
between pure point and singular continuous spectra. What are the analogs for power series? 

\medskip
We would like to thank Shmuel Agmon, John Garnett, Jean-Pierre Kahane, Rowan Killip, and Genadi Levin for useful discussions.

\section{Classical Natural Boundaries} \lb{s2}

In this section, we prove Theorem~\ref{T1.3}. The key will be a lemma of M.~Riesz \cite{MRie} used by many other authors
in the study of natural boundaries. The use of right limits and reflectionless power series can be viewed as a tool for
squeezing maximum benefit from Riesz's lemma.

Given a power series $f(z)=\sum_{n=0}^\infty a_n z^n$, where the $a_n$ obey \eqref{1.4}, we define
\begin{equation} \lb{2.1}
f_+^{(N)}(z)=\sum_{n=0}^\infty a_{n+N} z^n \qquad
f_-^{(N)}(z)=\sum_{n=-N}^{-1} a_{n+N} z^n
\end{equation}
so for $z\in\bbD\setminus\{0\}$,
\begin{equation} \lb{2.2}
f_+^{(N)}(z) + f_-^{(N)}(z)=z^{-N} f(z)
\end{equation}

Clearly, $f_-$ is defined and analytic on $\bbC\setminus\{0\}$, $f_+$ is defined initially on $\bbD$, but by \eqref{2.2},
has analytic continuation to any region that $f$ does.

\begin{theorem}[M.~Riesz's Lemma]\lb{T2.1} Suppose $\{a_n\}_{n=0}^\infty$ obeys \eqref{1.4} and that $f$ has an analytic
continuation to a neighborhood of $\bbD\cup S$ where
\begin{equation} \lb{2.3}
S=\{re^{i\theta}\mid 0<r\leq R,\, \alpha\leq\theta\leq\beta\}
\end{equation}
for some $R>1$ and $\alpha <\beta$. Then
\begin{equation} \lb{2.4}
\sup_{\substack{z\in S \\ N=0,1,2,\dots}} \abs{f_+^{(N)}(z)} <\infty
\end{equation}
\end{theorem}

\begin{remark} Riesz's lemma is usually in terms of $z^{-N} (f(z)-\sum_{n=0}^{N-1} a_n z^n)$, but this is eactly
$f_+^{(N)}(z)$.
\end{remark}

\begin{proof} While the result is classical and appears in many places (e.g., \cite{Remm}), we sketch the proof for completeness.

By comparing with a geometric series, for $z\in\bbD$ and all $N$\!,
\begin{equation} \lb{2.5}
\abs{f_+^{(N)}(z)}\leq (1-\abs{z})^{-1} \sup_n\, \abs{a_n}
\end{equation}
and similarly, for $z\in\bbC\setminus\ol{\bbD}$,
\begin{equation} \lb{2.6}
\abs{f_-^{(N)}(z)} \leq (1-\abs{z}^{-1})^{-1} \sup_n\, \abs{a_n}
\end{equation}

Let $\ti S$ have the form of $S$ with $\alpha,\beta$ replaced by $\ti\alpha,\ti\beta$ and $\ti\alpha <\alpha <\beta <\ti\beta$
so that $\ti S$ lies in the neighborhood of analyticity of $f$. Let $z_1=e^{i\ti\alpha}$, $z_2=e^{i\ti\beta}$, and define
\begin{equation} \lb{2.7}
g_\pm^{(N)}(z) = (z-z_1)(z-z_2) f_\pm^{(N)}(z)
\end{equation}

Clearly, \eqref{2.4} is implied by
\begin{equation} \lb{2.8}
\sup_{\substack{z\in\ti S \\ N=0,1,2,\dots}} \abs{g_+^{(N)}(z)}<\infty
\end{equation}
By the maximum principle, we need only check this on $\partial\ti S\setminus\{z_1,z_2\}$.

Because of \eqref{2.5} and the zeros of $g_+^{(N)}$, we have
\begin{equation} \lb{2.9}
\sup_{\substack{z\in\partial\ti S\cap\bbD \\ N=0,1,2,\dots}} \abs{g_+^{(N)}(z)} <\infty
\end{equation}
Similarly, by \eqref{2.6},
\begin{equation} \lb{2.10}
\sup_{\substack{z\in\partial\ti S\cap\bbC\setminus\ol{\bbD} \\ N=0,1,2,\dots }} \abs{g_-^{(N)}(z)} <\infty
\end{equation}
Since $\abs{z^{-N}}<1$ on $\bbC\setminus\ol{\bbD}$, \eqref{2.2} and \eqref{2.10} imply
\begin{equation} \lb{2.11}
\sup_{\substack{z\in\partial\ti S\cap\bbC\setminus\ol{\bbD} \\ N=0,1,2,\dots}} \abs{g_+^{(N)}(z)}<\infty
\end{equation}
proving \eqref{2.8}.
\end{proof}

\begin{proof}[Proof of Theorem~\ref{T1.3}] Suppose $a_{N_j+n}\to b_n$ for $n\in\bbZ$ and let $f_\pm$ be the functions in
\eqref{1.6}. Then, by estimating Taylor series, we have
\begin{equation} \lb{2.12}
f_+^{(N_j)}(z)\to f_+(z) \qquad
f_-^{(N_j)}(z) \to f_-(z)
\end{equation}
uniformly on compact subsets of $\bbD$ and $\bbC\cup\{\infty\}\setminus\ol{\bbD}$, respectively.

If $f$ has an analytic continuation to a neighborhood of $I=\{e^{i\theta}\mid\alpha_0 <\theta <\beta_0\}$, we can apply
Riesz's lemma for any $S$ of the form \eqref{2.3} with $\alpha_0 <\alpha<\beta<\beta_0$ and some suitable $R>1$ (depending
on $\alpha,\beta)$. Thus, by the Vitali convergence theorem, $f_+^{(N_j)}$ converges uniformly on $S$, so $f_+(z)$
has an analytic continuation to $S$. Moreover, by the analytic continuation of \eqref{2.2} to $S$ and $z^{-N_j}\to 0$
on $\bbC\setminus\bbD$, we see on $S\setminus\ol{\bbD}$,
\begin{equation} \lb{2.13}
f_+(z)+f_-(z)=0
\end{equation}
Thus, $\{b_n\}_{n=-\infty}^\infty$ is reflectionless across $I$.
\end{proof}

The proof shows that if $\Omega_I=\bbC\cup\{\infty\}\setminus(\partial\bbD\setminus I)$ and if $\calR$ is the family of
all right limits and $f_b(z)$ the function on $\Omega_I$ equal to $\sum_{n=0}^\infty b_n z^n$ on $\bbD$, then

\begin{theorem} \lb{T2.2} Fix $I$ an interval in $\partial\bbD$ and $A\in (0,\infty)$. Let $\calR$ be the set of
two-sided sequences reflectionless across $I$ with
\begin{equation} \lb{2.14}
\sup_{-\infty < n < \infty} \abs{b_n} \leq A
\end{equation}
Then, one has that, for any compact $K\subset\Omega_I$,
\begin{equation} \lb{2.15a}
\sup_{b\in\calR}\, \sup_{z\in K}\, \abs{f_b(z)} <\infty
\end{equation}
Moreover, $\calR$ is compact in the topology of uniform convergence on compacts of $\Omega_I$ and on
\begin{equation} \lb{2.15}
\{\{b_n\}_{n=-\infty}^{-1}\mid\{b_n\}_{n=-\infty}^\infty\in\calR\}
\end{equation}
$b_0,b_1,\dots$ are continuous functions of $\{b_n\}_{n=-\infty}^{-1}$. In fact, for any $k<\ell$, there is a homeomorphism
of $\{b_n\}_{n=-\infty}^k$ to $\{b_n\}_{n=\ell}^\infty$ by associating the two ends of a two-sided sequence.
\end{theorem}

\begin{proof} If $K\subset\bbD$ or $K\subset\bbC\setminus\ol{\bbD}$, one can use \eqref{2.5} or \eqref{2.6} for
$f_b$ and $\sup_n \abs{b_n}$. For $K$ straddling $I$, use the argument in the proof of Riesz's lemma and the fact that
$f_b$ is $\sum_{n=0}^\infty b_n z^n$ on $\bbD$ and $-\sum_{n=-\infty}^{-1} b_nz^n$ on $\bbC\setminus\ol{\bbD}$. This
proves \eqref{2.15a}.

Compactness then follows from Montel's theorem and the fact that, by these bounds, the set of reflectionless functions
is closed. The continuity is immediate if one notes that $\{b_n\}_{n=-\infty}^{-1}$ determines $f$ near $\infty$ and
$f$ near $0$ determines an $\{b_n\}_{n=0}^\infty$. The ``in fact'' statement uses that if $\{b_n\}_{n=0}^\infty$ is
reflectionless, so are its translates.
\end{proof}

Since Theorem~\ref{T1.3} says that if some right limit is not reflectionless, then $f$ has a natural boundary, and
since reflectionless $\{b_n\}_{n=-\infty}^\infty$ have one half determining the other half, we have the following,
which implies almost all the natural boundary results of this paper:

\begin{theorem}\lb{T2.3} Let $\{a_n\}_{n=0}^\infty$ be a bounded sequence with two right limits, $\{b_n\}_{n=-\infty}^\infty$
and $\{c_n\}_{n=-\infty}^\infty$, that obey
\begin{equation} \lb{2.14x}
b_0\neq c_0
\end{equation}
and either for some $k>0$ and all $j\geq k$ or for some $k<0$ and all $j\leq k$,
\begin{equation} \lb{2.15x}
b_j=c_j
\end{equation}
Then $\sum_{n=0}^\infty a_n z^n$ has a natural boundary.
\end{theorem}

\begin{remarks} 1. For example, if $k>0$, we look at the right limits $\{b_{n-k}\}_{n=-\infty}^\infty$ and
$\{c_{n-k}\}_{n=-\infty}^\infty$ which have the same $f_+$'s but unequal $f_-$'s.

\smallskip
2. In the next section, we extend this to conclude strong natural boundaries.
\end{remarks}

\section{Strong Natural Boundaries} \lb{s3}

In this section, we prove Theorem~\ref{T1.4}. We suppose that $f$ obeys \eqref{1.7} where $I=(\alpha,\beta)$.
Since \eqref{1.7} implies $\int_0^1 (\int_\alpha^\beta \abs{f(re^{i\theta})} \f{d\theta}{2\pi})\, dr <\infty$,
for a.e.\ $\theta_0$ in $(\alpha,\beta)$, we have $\int_0^1 \abs{f(re^{i\theta})}\, dr <\infty$. So, for a sequence
$\veps_n\downarrow 0$, we have $f\in E(S_n)$, where $S_n=\{re^{i\theta}\mid 0<r<1,\, \alpha+\veps_n <\theta <
\beta-\veps_n\}$ and $E(S_n)$ is the space introduced in \cite[Sect.~10.1]{Duren}. It follows that:
\begin{SL}
\item[(a)] $\lim_{r\uparrow 1} f(re^{i\theta}) = f(e^{i\theta})$ exists for a.e.\ $\theta\in(\alpha,\beta)$
(\cite[Thm.~10.3]{Duren}).
\item[(b)] For all $\veps$, $\int_{\alpha+\veps}^{\beta-\veps} \abs{f(e^{i\theta})}\f{d\theta}{2\pi} <\infty$
(\cite[Thm.~10.3]{Duren}).
\item[(c)] For every $\veps>0$,
\begin{equation} \lb{3.1}
\lim_{r\uparrow 1} \int_{\alpha+\veps}^{\beta-\veps} \abs{f(re^{i\theta}) - f(e^{i\theta})}\, \f{d\theta}{2\pi} =0
\end{equation}
\item[(d)] If we define
\begin{equation} \lb{3.2}
F(z)=\int_{\alpha+\veps}^{\beta-\veps} f(e^{i\theta}) (e^{i\theta}-z)^{-1}\, \f{d\theta}{2\pi i}
\end{equation}
then $F$ is analytic in $\bbC\setminus\{e^{i\theta}\mid \alpha+\veps <\theta < \beta-\veps\}$ and
\begin{equation} \lb{3.3}
\lim_{r\uparrow 1}\, F(re^{i\theta}) -\lim_{r\downarrow 1}\, F(re^{i\theta}) =f(e^{i\theta})
\end{equation}
(follows from \cite[Thm.~10.4]{Duren}).
\end{SL}

We also need the following Painlev\'e-type theorem, which follows easily from Morera's theorem:
\begin{SL}
\item[(e)] If $f_+$ is analytic in $\bbD$, $f_-$ in $\bbC\setminus\ol{\bbD}$, and
\begin{equation} \lb{3.4}
\sup_{0<r<1} \int_\alpha^\beta \abs{f_+(re^{i\theta})}\, \f{d\theta}{2\pi} +
\sup_{1<r<2} \int_\alpha^\beta \abs{f_-(re^{i\theta})}\, \f{d\theta}{2\pi} <\infty
\end{equation}
and if for a.e.\ $\theta\in(\alpha,\beta)$,
\begin{equation} \lb{3.5}
f_+ (e^{i\theta}) = f_-(e^{i\theta})
\end{equation}
then there is $G$ analytic in $\bbC\setminus [\partial\bbD\setminus\{e^{i\theta}\mid\alpha<\theta<\beta\}]$ so that
$G=f_+$ on $\partial\bbD$ and $f_-$ on $\bbC\setminus\ol{\bbD}$.
\end{SL}

\begin{proof}[Proof of Theorem~\ref{T1.4}] Let $F$ be given by \eqref{3.2} and define
\begin{equation} \lb{3.6}
b_n = \int_{\alpha+\veps}^{\beta-\veps} e^{-in\theta} f(e^{i\theta})\, \f{d\theta}{2\pi}
\end{equation}
Then, by expanding $(e^{i\theta}-z)^{-1}$ in suitable geometric series, the Taylor expansion of $F$ near zero is
$\sum_{n=0}^\infty b_n z^n$ and near $\infty$ is $\sum_{n=-\infty}^{-1} b_n z^n$. Moreover, by the Riemann--Lebesgue
lemma and (b) above,
\begin{equation} \lb{3.7}
\lim_{\abs{n}\to\infty}\, b_n =0
\end{equation}

Let $c_n=a_n-b_n$ and
\begin{equation} \lb{3.8}
f_+(z)=\sum_{n=0}^\infty c_n z^n
\end{equation}
Then, by (a), for a.e.\ $\theta\in(\alpha+\veps,\beta-\veps)$,
\begin{equation} \lb{3.9}
\lim_{r\uparrow 1} f_+ (re^{i\theta}) = f(e^{i\theta}) -\lim_{r\uparrow 1}\, F(re^{i\theta})
\end{equation}
and if
\begin{equation} \lb{3.10}
f_-(z) = -\sum_{n=-\infty}^{-1} b_n z^n
\end{equation}
then, for a.e.\ $\theta\in(\alpha+\veps,\beta-\veps)$,
\begin{equation} \lb{3.11}
\lim_{r\downarrow 1}\, f_-(z) =-\lim_{r\downarrow 1}\, F(re^{i\theta})
\end{equation}

It follows that \eqref{3.5} holds, so $f_+$ has a classical analytic continuation across $(\alpha+\veps,
\beta-\veps)$. By Theorem~\ref{T1.3}, every right limit of $c_n$ is reflectionless on $(\alpha+\veps, \beta-\veps)$.

But, by \eqref{3.7}, the right limits of $c_n$ and $a_n$ are the same! Thus, each right limit of $a_n$ is reflectionless
on each $(\alpha+\veps,\beta-\veps)$, and so on $(\alpha,\beta)$.
\end{proof}

\begin{theorem}\lb{T3.1} If the hypotheses of Theorem~\ref{T2.3} hold, then $\sum_{n=0}^\infty a_n z^n$ has a
strong natural boundary on $\partial\bbD$.
\end{theorem}

\begin{example} \lb{E3.2} The Rudin--Shapiro \cite{Rudin,Shap} sequence is defined by defining polynomials $P_n$ and
$Q_n$ recursively by $P_0(z)=Q_0(z)=1$, $P_{n+1}(z) = P_n(z)+z^{2^n} Q_n(z)$, $Q_{n+1}(z)=P_n(z)- z^{2^n} Q_n(z)$. As
power series with bounded coefficients, $\lim\, P_n(z)$ exists and defines a series $f(z)=\sum_{n=0}^\infty a_n z^n$,
where each $a_n$ is $=1$ or $-1$. By using Szeg\H{o}'s theorem, Brillhart \cite{Brill} proved this function had a natural
boundary. Here is an elementary direct proof. $P_{n+1}$ is $P_{n-1} Q_{n-1} P_{n-1} (-Q_{n-1})$. Taking right limits at
the end of the $P_{n-1}$ in $P_{n-1} Q_{n-1}$ and in $P_{n-1}(-Q_{n-1})$ yields right limits which agree at negative
index but have opposite signs at positive index.
\qed
\end{example}

\section{Gap Theorems} \lb{s4}

In this section, we'll prove Theorem~\ref{T1.6} and resolve the apparent contradiction to P\'olya \cite{Polya}--Erd\"os
\cite{Erdos}. The following proof shows the power of reflectionless theorems.

\begin{proof}[Proof of Theorem~\ref{T1.6}] By compactness, we can find a subsubsequence, $n_{j_\ell}$, so $\lim_{\ell\to\infty}
a_{n_{j_\ell}+k}=b_k$ exists for all $k$ and $b_0\neq 0$, $b_k=0$ for $k<0$. For such a right limit, $f_-(z)=0$, but $f_+(0)
=b_0\neq 0$. Thus, $f_-$ cannot be an analytic continuation of $f_+$ through any $I$, and this $\{b_n\}_{n=-\infty}^\infty$
is not reflectionless on any $I$. Theorem~\ref{T1.4} completes the proof.
\end{proof}

It is simplest to resolve the apparent contradiction with \cite{Polya,Erdos} in the context of an example.

\begin{example}\lb{E4.1} Let $U=\cup_{j=2}^\infty \{n\mid j!\leq n\leq j! + j\}$. We want to consider bounded power series
with $a_n=0$ if $n\in U$\!. If $a_{j!+j+1}=1$ for $j\geq 2$ and $a_n$ is arbitrary but bounded for $n\notin U\cup
\{j!+j+1\}_{j=2}^\infty$, then $f(z)$ has a strong natural boundary on $\partial\bbD$. The zero values may be only on
$U$\!, which is a set with zero density.

On the other hand, \cite{Polya,Erdos} say that since $U$ does not have density $1$, there must be $\{a_n\}_{n=1}^\infty$
with $a_n=0$ for $n\in U$ so that $\partial\bbD$ is not a natural boundary.

The resolution is that our natural boundary examples have hard edges, that is, $a_n$ jumps at the edges of $U$\!, while
the examples of \cite{Erdos} have soft edges. $\bbZ_+\setminus U$ has longer and longer nonzero intervals and Erd\"os'
examples ramp up slowly and down slowly to be $1$ in the center of these intervals. It is easy to see that these examples
have right limits which are constant, and so reflectionless!

These examples of Erd\"os are reminiscent of the sparse potentials of Molchanov \cite{Mol} and Remling \cite{Rem}, where
approximate solitons are placed in between long gaps. In fact, given the chronology, we should say the examples of \cite{Mol,Rem}
are reminiscent of Erd\"os \cite{Erdos}!
\qed
\end{example}

As noted by Agmon \cite{Agmon1, Agmon}, it isn't important that $a_{N_j+n}\to 0$ as $j\to 0$ from gaps, only that its $\limsup$
decays exponentially fast. Using right limits, this is easy to see since

\begin{theorem}\lb{T4.2} Let $\{b_n\}_{n=-\infty}^\infty$ be a two-sided series which is reflectionless on some interval,
$I\subset\partial\bbD$. Suppose that for some $C,D>0$, we have that
\begin{equation} \lb{4.1}
\abs{b_n} \leq Ce^{-Dn} \qquad\text{for } n>0
\end{equation}
Then $b_n\equiv 0$.
\end{theorem}

\begin{proof} By \eqref{4.1}, $f_+(z)$ has an analytic continuation to the circle $\{z\mid\abs{z} <e^D\}$. Since $f_+ =
-f_-$ in a neighborhood of $I$, we conclude that $f_+$ defines an entire function. Since $\abs{f_-(z)}\to 0$ as $\abs{z}\to
\infty$, $f_+\equiv 0$, so $f_-\equiv 0$ also, and then $b_n\equiv 0$.
\end{proof}

This immediately implies the following extension of Theorem~\ref{T1.6}:

\begin{theorem} \lb{T4.3} Suppose $\{a_n\}_{n=0}^\infty$ obeys \eqref{1.4} and there exists $n_j\to\infty$ so that for
some $C,D>0$ and for all $k<0$,
\begin{gather}
\limsup_{j\to\infty}\, \abs{a_{n_j+k}} \leq Ce^{-D\abs{k}} \lb{4.2} \\
\liminf \, \abs{a_{n_j}} >0 \lb{4.3}
\end{gather}
Then $f(z)=\sum_{n=0}^\infty a_n z^n$ has a strong natural boundary on $\partial\bbD$.
\end{theorem}

\begin{proof} By compactness, there exists a right limit, so \eqref{4.1} holds for $n<0$ and $b_0\neq 0$. If
$\{b_n\}_{n=-\infty}^\infty$ is reflectionless on $I$, then $\{b_{-n}\}_{n=-\infty}^\infty$ is reflectionless on
$\bar I=\{z\mid\bar z\in I\}$, so Theorem~\ref{T4.2} implies $b_n\equiv 0$ if this right limit is reflectionless.
But $b_0\neq 0$.
\end{proof}

\section{Szeg\H{o}'s Theorem} \lb{s5}

In this section, we'll prove

\begin{theorem}\lb{T5.1} Let $\sum_{n=0}^\infty a_n z^n$ where the values of $\{a_n\}$ lie in a finite set, $F$\!. Then
either $f(z)=\sum_{n=0}^\infty a_n z^n$ has a strong natural boundary or $a_n$ is eventually periodic, in which case
$f$ is a rational function with poles at roots of unity.
\end{theorem}

\begin{remark} This result for ordinary natural boundaries is due to Szeg\H{o} \cite{Sz1922}. That $f$ is unbounded on any
sector is due to Duffin--Schaeffer \cite{DuSc}, and that there is a strong natural boundary is a result of Boas \cite{Boas}.
This is an analogy of spectral theory results of Kotani \cite{Kot89} and Remling \cite{Rem}. One could use an argument of Kotani
and our reflectionless machinery to prove Theorem~\ref{T5.1}, but instead we'll borrow part of Boas' argument and note
one could use that to find an alternate proof of the spectral theory results.
\end{remark}

\begin{proof} Let $V$ be the finite set of possible values of $a_n$. Suppose $\{a_n\}$ is not eventually periodic.
Fix $p=1,2,\dots$ and consider $p$ blocks, $\{a_j\}_{j=\ell p+1}^{(\ell +1)p}$ for $\ell=0,1,2,\dots$. Since there are
only $(\#V)^p$ possible $p$ blocks, some value must recur, that is, there exist $Q_p >P_p$ so that
\begin{equation}\lb{5.1}
a_{Q_p+j} = a_{P_p+j} \qquad j=1, \dots, p
\end{equation}
If \eqref{5.1} holds for all $j\geq 1$, then for $k=P_p + j\geq P_p+1$, we have
\[
a_{(Q_p-P_p)+k} = a_k \qquad k=1, \dots
\]
that is, $a$ is eventually periodic. Since we are assuming the contrary, there is $L_p\geq p+1$, so $a_{Q_p + L_p} \neq
a_{P_p+L_p}$.

Let $N_p=P_p + L_p\to\infty$ as $p\to\infty$ since $L_p\geq p$ and $M_p =Q_p + L_p >N_p+L$. Then
\begin{align}
a_{N_p}+j &= a_{M_p}+j \qquad j=-p, \dots, -1 \lb{5.2} \\
a_{N_p} &\neq a_{M_p} \lb{5.3}
\end{align}
By compactness, we get right limits, $b,c$, obeying \eqref{2.15}/\eqref{2.14x}, so by Theorem~\ref{T3.1}, $\sum_{n=0}^\infty
a_n z^n$ has a strong natural boundary.
\end{proof}

We note the following extension of Szeg\H{o}'s theorem which appears in Bieberbach \cite{Bie} (who only proved classical natural boundary):

\begin{theorem} \lb{T5.1} Let $\{a_n\}_{n=0}^\infty$ be a bounded sequence with finitely many limit points. Then either
$\sum_{n=0}^\infty a_n z^n$ has a strong natural boundary or there is a periodic sequence, $c_n$, with
\begin{equation} \lb{5.4}
\abs{a_n-c_n} \to 0 \qquad\text{as } n\to\infty
\end{equation}
\end{theorem}

\begin{proof} Let $V$ be the finite set of limit points. Let $\gamma=\f12 \min_{x,y\in V,\, x\neq y} \abs{x-y}$.
Eventually, for all $n$, there is $c_n\in V$ with $\abs{a_n-c_n}\leq \gamma$. It follows that \eqref{5.4} holds. If $c_n$
is eventually periodic, it can be modified for small $n$ to be periodic and \eqref{5.4} still holds.

If $c_n$ is not eventually periodic, by the above, it has a right limit which is not reflectionless. But $a_n$ and $c_n$
have the same right limits.
\end{proof}

\section{Random Power Series} \lb{s6}

In this section, we'll prove Theorem~\ref{T1.7} as well as

\begin{theorem}\lb{T6.1} Let $\{a_n(\omega)\}_{n=0}^\infty$ be a sequence of independent random variables so that
\begin{SL}
\item[{\rm{(i)}}] $\sup_{n,\omega} \abs{a_n(\omega)} =K<\infty$
\item[{\rm{(ii)}}] For some sequence $n_j\to\infty$,
\begin{equation} \lb{6.1}
\limsup_{j\to\infty}\, [\bbE(\abs{a_{n_j}(\omega)}^2) - \abs{\bbE(a_{n_j}(\omega))}^2] >0
\end{equation}
Then for a.e.\ $\omega$, $\sum_{n=0}^\infty a_n(\omega) z^n$ has a strong natural boundary.
\end{SL}
\end{theorem}

\begin{remark} $\bbE$ is expectation. We'll use $\bbP$ for probability and $\bbV$ for variation, so \eqref{6.1} is
$\limsup_{j\to\infty} \bbV(a_{n_j}(\omega)) >0$.
\end{remark}

\begin{lemma}\lb{L6.2}
\begin{SL}
\item[{\rm{(i)}}] For any $K$ and $m$, there exists $K_m >0$ so that for any random variable $f$ with $\norm{f}_\infty \leq K$\!,
there exists $z\in\bbC$ so that
\begin{equation} \lb{6.2}
\bbP\biggl( \abs{f(\omega)-z}\leq \f{1}{m}\biggr) \geq K_m
\end{equation}

\item[{\rm{(ii)}}] For any $K$\!, $m$, and $\sigma >0$, there exists $\wti K_m >0$ so that for any random variable $f$ with
$\norm{f}_\infty \leq K$ and $\bbV(f)\geq\sigma$, there exists $z,w\in\bbC$ with
\begin{equation} \lb{6.3}
\abs{z-w} \geq\biggl( \f{\sigma}{2}\biggr)^{1/2}
\end{equation}
so that
\begin{equation} \lb{6.4}
\bbP\biggl(\abs{f(\omega)-z} \leq \f{1}{m}\biggr) \geq \wti K_m \qquad
\bbP\biggl( \abs{f(\omega)-w}\leq \f{1}{m}\biggr) \geq \wti K_m
\end{equation}
\end{SL}
\end{lemma}

\begin{proof} (i) Cover $\{z\mid\abs{z}\leq K\}$ by disks of radius $\f1m$. By compactness (or simple geometry), one can arrange
for a finite number $N_m$. One of these disks must have probability at least $K_m=(N_m)^{-1}$. \eqref{6.2} holds for the center
of that disk.

\smallskip
(ii) For any $c>0$ and bounded random variable $g$,
\begin{equation} \lb{6.5}
\bbE(\abs{g}^2) \leq c^2\bbP(\abs{g}\leq c) + \norm{g}_\infty^2 \bbP(\abs{g}\geq c)
\end{equation}
Thus, for any $\alpha$ with $\abs{\alpha}\leq\norm{f}_\infty$,
\begin{equation} \lb{6.6}
\var(f) \leq c^2 + (2\norm{f}_\infty)^2 \bbP(\abs{f-\alpha}>c)
\end{equation}
Picking $c\leq (\f{\sigma}{2})^{1/2}$ with $\sigma=\var(f)$, we see that
\begin{equation} \lb{6.7}
\bbP(\abs{f-\alpha}>c)\geq \f{1}{8\norm{f}_\infty^2} \var (f)
\end{equation}

Use (i) to find $z$ so $\bbP(\abs{f(\omega)-z}\leq \f{1}{m})\geq K_m$ where, if necessary, $m$ is increased so
$\f{1}{m} \leq (\f{\sigma}{2})^{1/2}$. By repeating (i) using $\bbP(\abs{f-z}\geq (\f{\sigma}{2})^{1/2}) \geq
\f{1}{8\norm{f}_\infty} \sigma$, we get a $w$ outside the disk $\{\zeta\mid\abs{\zeta-z}\leq(\f{\sigma}{2})^{1/2}\}$
so that $\bbP(\abs{f(\omega)-w}\leq\f{1}{m}) \geq \wti K_m$ for some $\wti K_m\leq K_m$.
\end{proof}

\begin{proof}[Proof of Theorem~\ref{T6.1}] Pick $\sigma >0$ and $n_j\to\infty$ so that $\bbV(a_{n_j}(\omega))\geq\sigma$
for all $j$ and so that $n_{j+1} \geq n_j$ and $n_{j+1}-n_j\to\infty$. For each $n\neq n_j$, for all $j$, use (i) of
Lemma~\ref{L6.2} to pick $z_n^{(m)}$ so that
\begin{equation} \lb{6.8}
\bbP\biggl(\abs{a_n(\omega)-z_n^{(m)}} \leq \f{1}{m}\biggr) \geq K_m
\end{equation}

By compactness and using the diagonalization trick, one can pass to a subsequence of the $n_j$, which we'll still
denote by $n_j$,
so for all $m$ and all $k\neq 0$, $z_{n_j+k}^{(m)}\to z_k$ as $j\to\infty$.

By using (ii) of Lemma~\ref{L6.2}, find $w_j^{(m)}, \zeta_j^{(m)}$ so that
\begin{equation} \lb{6.9}
\begin{aligned}
\bbP\biggl( \abs{a_{n_j}(\omega)-w_j^{(m)}} \leq \f{1}{m}\biggr) &\geq\wti K_m \\
\bbP\biggl( \abs{a_{n_j}(\omega)-\zeta_j^{(m)}} \leq \f{1}{m}\biggr) & \geq \wti K_m
\end{aligned}
\end{equation}
and $\abs{w_j^{(m)} - \zeta_j^{(m)}} \geq (\f{\sigma}{2})^{1/2}$. Again, by passing to a subsequence, we may assume
that $w_j^{(m)}\to w$, $\zeta_j^{(m)}\to \zeta$ where, of course, $\abs{\zeta-w} \geq (\f{\sigma}{2})^{1/2}$.

By the Borel--Cantelli lemma, for each $m$ and each $Q$,
\begin{equation} \lb{6.10}
\abs{a_{n_j}(\omega) - w_j^{(m)}} \leq \f{1}{m} \qquad \abs{a_{n_j+k}(\omega) - z_{n_j+k}^{(m)}} \leq \f{1}{m}
\end{equation}
for all $k$ with $0 < \abs{k}\leq Q$, occurs infinitely often for a.e.\ $\omega$, and the same for $\zeta_j^{(m)}$. Thus, for
a.e.\ $\omega$, the right limits include $b,c$ with
\begin{equation} \lb{6.11}
b_k=c_k=z_k \text{ for } k\neq 0 \qquad b_0 =w \qquad c_0 = \zeta\neq w
\end{equation}

It follows from Theorem~\ref{T3.1} that $a_n(\omega)$ has a strong natural boundary for all $\omega$.
\end{proof}

\begin{remark} The use of the Borel--Cantelli lemma to get nonreflectionless limits (in the context of spectral theory of
CMV matrices) is taken from work of Breuer--Ryckman--Zinchenko \cite{BRZ}.
\end{remark}

For our proof of Theorem~\ref{T1.7}, we need two lemmas whose proof we defer to the end of the section.

\begin{lemma}\lb{L6.3} \ Let $\alpha <\beta$ in $(-\pi,2\pi)$ with $\abs{\beta-\alpha}<2\pi$. Then $\{\{a_n\}_{n=0}^\infty
\mid \sup_n \abs{a_n}\leq A$ for all $\veps >0$; $\sup_{0<r<1} (\int_{\alpha+\veps}^{\beta-\veps} \abs{\sum_{n=0}^\infty
a_n r^n e^{in\theta}} \f{d\theta}{2\pi}) <\infty \}$ is a measurable set {\rm{(}}in the product topology{\rm{)}} invariant
under $a_n\to a_{n+1}$.
\end{lemma}

\begin{lemma}\lb{L6.4} Let $\mu$ be a probability measure on $\{\{a_n\}_{n=-\infty}^\infty \mid \sup_n \abs{a_n}\leq A\}$
that defines an ergodic invariant stochastic process. Let $\calR_\mu$ be the set of two-sided series $\{b_n\}_{n=-\infty}^\infty$
so that $\{b_n\}_{n=-\infty}^\infty$ is a right limit of $\{a_n\}_{n=0}^\infty$ with positive probability. Then $\calR_\mu$ is
the support of the measure $\mu$.
\end{lemma}

\begin{proof}[Proof of Theorem~\ref{T1.7}] By Lemma~\ref{L6.3}, Theorem~\ref{T1.4}, and ergodicity (letting $\alpha,\beta$
run through rational multiples of $2\pi$), either $\sum_{n=0}^\infty a_n z^n$ has a strong natural boundary with probability
$1$ or there is an interval, $I$, in $\partial\bbD$ so that with probability $1$, all right limits of $\sum_{n=0}^\infty
a_n z^n$ are reflectionless across $I$. In that case, by Lemma~\ref{L6.4}, all $\{b_n\}_{n=-\infty}^\infty$ in the support
of $\mu$ are reflectionless across $I$. By Theorem~\ref{T2.2}, on $\supp(\mu)$, $b_0$ is a continuous, and so measurable, function
of $\{b_n\}_{n=-\infty}^{-1}$. Thus, the process is deterministic (in the strong sense of there being a continuous, rather than
merely a measurable function).
\end{proof}

\begin{proof}[Proof of Lemma~\ref{L6.3}] By Theorem~10.1 of Duren \cite{Duren}, the space $E$ (sector) can be defined by a
countable family of approximating curves. Thus, letting $z_1^{(n)}=\exp(i(\alpha+\f{1}{n}))$, $z_2^{(n)}=\exp(i(\beta-\f{1}{n}))$,
the $\sup$ condition can be replaced by
\begin{equation} \lb{6.12}
\sup_{m,n}\left( \int_{\alpha+\f{1}{n}}^{\beta-\f{1}{n}} \left. \biggl| (z-z_1^{(n)})(z-z_2^{(n)})
\sum_{n=0}^\infty a_n z^n \biggr| \right|_{z=(1-\f{1}{m})e^{i\theta}} \f{d\theta}{2\pi} \right) <\infty
\end{equation}
Since it is described by a countable $\sup$ of uniformly convergent sums, this set is clearly measurable.

If $f(z)=\sum_{n=0}^\infty a_n z^n$ and $\ti f(z)=\sum_{n=0}^\infty a_{n+1} z^n$, then
\begin{equation} \lb{6.13}
\ti f(z)=z^{-1} (f(z)-a_0 z)
\end{equation}
and finiteness of the $\sup$ in \eqref{6.12} implies finiteness of the $\sup$ for $\ti f$ replacing $f$. This proves the
claimed invariance.
\end{proof}

\begin{proof}[Proof of Lemma~\ref{L6.4}] If $b_n$ is in $\supp(d\mu)$, then for any $\veps$, $N,\mu(\{a_n\mid
\abs{a_{N+n}-b_n}<\veps,\, n=0,\pm 1, \dots, \pm N\})>0$. By the ergodic theorem for a.e.\ $\omega$,
\begin{equation} \lb{6.14}
\lim_{M\to\infty} \, \f{1}{M}\, \#\{j<M\mid \abs{a_{n+j+N}-b_j} <\veps, \, n=0,\pm 1, \dots, \pm N\} >0
\end{equation}
so for a.e.\ $\omega$, there exists $N_\ell \to\infty$ with $\abs{a_{n+N_\ell}-b_n}<\veps$ for all $n$ with $\abs{n}<N$.
By a diagonalization trick, $b_n$ is a right limit for a.e.\ $\omega$.

If $b_n$ is not in the support of $\mu$, pick $N$ and $\veps$ so that $\mu(\{a_n\mid \abs{a_{N+n}-b_n}\leq \veps, \,
\abs{n}\leq N\})=0$. By translation invariance for all $k$,
\begin{equation} \lb{6.15}
\mu(\{a_n\mid \abs{a_{N+k+n}-b_n} \leq \veps, \, n\leq N\})=0
\end{equation}
which implies $b_n$ is not a right limit with probability $1$.
\end{proof}

\section{Hecke's Example} \lb{s7}

Motivated by a spectral theory result of Damanik--Killip \cite{DK}, we prove the following that includes Theorem~\ref{T1.8}
and so, Hecke's example.

\begin{theorem} \lb{T7.1} Let $T\colon\partial\bbD\to\partial\bbD$ be a homeomorphism so that for any $e^{i\theta}\in
\partial\bbD$, $\{T^k (e^{i\theta})\mid k=0,1,\dots\}$ is dense in $\partial\bbD$. Let $f\colon\partial\bbD\to\bbC$
be a bounded and piecewise continuous function with a finite number of discontinuities so that at, at least, one discontinuity,
the right and left limits exist and are not equal. Then for any $e^{i\theta}\in\partial\bbD$, $\sum_{n=0}^\infty
f(T^n (e^{i\theta}))z^n$ has a strong natural boundary on $\partial\bbD$.
\end{theorem}

\begin{remark} For other papers on other extensions of Hecke's example, see \cite{Salem,Newman,Mordell,Schwarz,Meijer,
Popken,CarKem}.
\end{remark}

\begin{proof} By rotating, we suppose $e^{i\theta}=1$ is a point of a discontinuity with $\lim_{\theta\downarrow 0}
f(e^{i\theta})=r\neq s=\lim_{\theta\uparrow 0} f(e^{i\theta})$. By the density of any orbit for any $e^{i\theta}$,
we can find $n_j\to\infty$, so $T^{n_j}(e^{i\theta})\to 1$ with $T^{n_j}(e^{i\theta})=e^{i\psi_j}$ with $-\pi < \psi_j <0$
and $m_j\to\infty$, so $T^{m_j}(e^{i\theta})\to 1$ and $T^{m_j}(e^{i\theta})=e^{i\eta_j}$ with $\pi >\eta_j>0$ (for find
$n_{j+1}>n_j$ so $\abs{T^{n_{j+1}}(e^{i\theta})-e^{-i/(j+1)}}\leq (j+1)^{-2}$).

Thus, $a_{n_j}\to s$ and $a_{m_j}\to r$. On the other hand, look at the orbit $\{T^\ell(1)\mid\ell=1,2,\dots\}$.
Since this orbit is dense, the $T^\ell(1)$ must be distinct, so for some $L$ and all $\ell >L$, $T^\ell(1)$ must be
a point of continuity. Thus, for $\ell >L$, $a_{n_j+\ell}-a_{m_j+\ell}\to 0$. The hypotheses of Theorem~\ref{T2.3}
hold for the right limits defined by $a_{n_j+n}$ and $a_{m_j+n}$. So, by Theorem~\ref{T3.1}, we have a strong natural
boundary.
\end{proof}

\section{Baire Genericity} \lb{s8}

\begin{proof}[Proof of Theorem~\ref{T1.9}] Let $\{a_n\}_{n=0}^\infty\in\Omega^\infty$. Pick two distinct points,
$z_0,z_1\in\Omega$. For any $\ell$, define
\[
a_n^{(\ell)} = \begin{cases}
a_n & n\leq \ell \\
z_0 & n=\ell+k! \text{ for } k=1,2,\dots \\
z_1 & n=\ell +m, \,\,  m\geq 1,\,\, m\neq k! \text{ for any } k
\end{cases}
\]
Then, by the gap theorem, $\{a_n^{(\ell)}\}$ has a natural boundary on $\partial\bbD$, indeed, by Weierstrass' original
direct arguments. But $\lim_{\ell\to\infty} a_n^{(\ell)}=a_n$ for each fixed $n$, so the set in the theorem is dense in
the weak topology.

For every rational multiple $\alpha,\beta$ of $2\pi$ with $\alpha <\beta$, and every $K=1,2,\dots$, and any $n=1,\dots$,
let $F_{\alpha,\beta,K,n} = \{\{a_n\}_{n=1}^\infty\in\Omega^\infty \mid f(z)=\sum_{n=0}^\infty a_n z^n$ has an analytic
continuation to $\{z\mid \abs{z}<1+n^{-1},\, \alpha <\arg(z)<\beta\}$ with $\abs{f(z)}\leq K$ there$\}$. It is easy to see that
\begin{SL}
\item[(i)] Each $F_{\alpha,\beta,K,n}$ is closed in the topology of pointwise convergence of the $a_n$'s (by the Vitali
theorem).
\item[(ii)] $\cup F_{\alpha,\beta,K,n}$ is the set of power series for which $\partial\bbD$ is not a natural boundary.
\end{SL}

Thus, the complement of the set in the theorem is an $F_\delta$, so the set is a $G_\delta$.
\end{proof}

\bigskip

\end{document}